\documentclass[reqno, 11pt]{amsart}

\usepackage[text={155mm,235mm},centering]{geometry}
\input diagxy
\input xy

\usepackage[active]{srcltx} % SRC Specials

\newtheorem{thm}{Theorem}[section]
\newtheorem{cor}[thm]{Corollary}
\newtheorem{lem}[thm]{Lemma}
\newtheorem{prop}[thm]{Proposition}

\theoremstyle{definition}

\theoremstyle{property}

\theoremstyle{remark}
\newtheorem{rem}[thm]{Remark}

\numberwithin{equation}{section}
\usepackage[mathscr]{eucal}
\usepackage[all]{xy}
\usepackage{mathrsfs}
\usepackage{xypic}
\usepackage{amsfonts}
\usepackage{amsmath}
\usepackage{amsthm,bm}
\usepackage{amssymb,tikz-cd}
\usepackage{latexsym}
\usepackage{tabularx}
\usepackage{graphicx}
\usepackage{pict2e}
\usepackage[pagebackref]{hyperref}
\usepackage{tikz}
\usepackage{textcomp}
\usepackage[scr=rsfs]{mathalfa}
\definecolor{ceruleanblue}{rgb}{0.16, 0.32, 0.75}
\hypersetup{colorlinks=true,allcolors=ceruleanblue}
\allowdisplaybreaks

\begin{document}

\title[Leray-Hirsch theorem and blow-up formula for Dolbeault cohomology]{Leray-Hirsch theorem and blow-up formula for Dolbeault cohomology}

\author{Lingxu Meng}
\address{Department of Mathematics, North University of China, Taiyuan, Shanxi 030051,  P. R. China}
\email{menglingxu@nuc.edu.cn}%
\date{\today}


% -----------------------------------------------------------

\begin{abstract}
  We prove  a theorem of Leray-Hirsch type  and give an explicit blow-up formula for Dolbeault cohomology on (\emph{not necessarily compact}) complex manifolds. We give applications to  strongly $q$-complete manifolds and the $\partial\bar{\partial}$-lemma.
\end{abstract}

% -----------------------------------------------------------
\maketitle
% -----------------------------------------------------------

%==================================

\section{Introduction}
Dolbeault cohomology is an algebraic invariant on complex manifolds, which is related to the complex structures of complex manifolds. In general cases, the computation of Dolbeault cohomology is much more difficult than that of de Rham one. One reason is that the property of a sheaf germs of holomorphic forms is more complicated than that of a constant sheaf. We focus on  the Dolbeault cohomologies of fiber bundles and blow-ups.

The Leray-Hirsch theorem is a classical result in algebraic topology, which is used to calculate the singular (or de Rham) cohomology of fiber bundles.
Cordero, L. \cite{CFGU} obtained a version of this theorem on Dolbeault cohomology by use of Borel's spectral sequence on holomorphic fiber bundle.
By Cordero's result, Rao, S., Yang, S. and Yang, X.-D. \cite{RYY} proved  the projective bundle formula for Dolbeault cohomology on  compact complex manifolds. Angella, D., Suwa, T., Tardini, N. and  Tomassini, A. \cite{ASTT} and Stelzig, J. \cite{St2}  got the same formula for projective bundles in  different ways.
%Cordero's proof is much sophisticated.
The following  theorem is weaker than Cordero's and we will give a simple proof.
\begin{thm}\label{L-H}
Let $\pi:E\rightarrow X$ be a holomorphic fiber bundle with compact fibers over a connected complex manifold $X$. Assume that there exist classes $e_1$, $\dots$, $e_r$ of pure degrees in $ H_{\bar{\partial}}^{*,*}(E)$, such that,  their restrictions $e_1|_{E_x}$, $\dots$, $e_r|_{E_x}$ to $E_x$ is a basis of  $H_{\bar{\partial}}^{*,*}(E_x)$ for every $x\in X$, where $H^{*,*}_{\bar{\partial}}(\bullet)=\bigoplus\limits_{p,q\geq 0}H^{p,q}_{\bar{\partial}}(\bullet)$.  Then
\begin{displaymath}
\sum_{i=1}^r\pi^*(\bullet)\cup e_i:\bigoplus\limits_{i=1}^rH_{\bar{\partial}}^{p-u_i,q-v_i}(X) \rightarrow H_{\bar{\partial}}^{p,q}(E)
\end{displaymath}
is an isomorphism for any $p$, $q$, where $e_i$ is of degree $(u_i,v_i)$ for $1\leq i\leq r$.
\end{thm}

Rao, S. et. al. \cite{RYY} proved that, there exists an isomorphism for Dolbeault cohomology of blow-ups of compact complex manifolds by the cohomology of sheaves and expressed it explicitly for the one in the Fujiki class $\mathcal{C}$. They also provided some beautiful applications to bimeromorphic geometry. Shortly after \cite{RYY} was posted on arXiv, Angella, D. et. al. \cite{ASTT} obtained the similar result under the assumption that the center admits a holomorphically contractible neighborhood by \u{C}ech-Dolbeault cohomology and considered the orbifold cases under some restrictions.  Stelzig, J. \cite{St2} studied the blow-up formula on the level of  the double complex for  compact complex manifolds up to a suitable notation of quasi-isomorphism. We give an explicit expression of the blow-up formula  and prove that it holds on general (\emph{not necessarily compact}) complex manifolds for the Dolbeault cohomology.

\begin{thm}\label{main}
Let $\pi:\widetilde{X}_Z\rightarrow X$ be the blow-up of a connected complex manifold $X$ along a connected complex submanifold $Z$ of codimension $r$. Let $E$ be the exceptional divisor and $i_E:E\rightarrow\widetilde{X}_Z$ the inclusion. Then
\begin{equation}\label{b-u-m}
\pi^*+\sum_{i=1}^{r-1}i_{E*}\circ (h^{i-1}\cup)\circ (\pi|_E)^*
\end{equation}
gives an isomorphism
\begin{displaymath}
H_{\bar{\partial}}^{p,q}(X)\oplus \bigoplus_{i=1}^{r-1}H_{\bar{\partial}}^{p-i,q-i}(Z)\tilde{\rightarrow} H_{\bar{\partial}}^{p,q}(\widetilde{X}_Z)
\end{displaymath}
for any $p$, $q$, where $\pi|_E:E\rightarrow Z$ is naturally viewed as the projective bundle $E=\mathbb{P}(N_{Z/X})$ associated to the normal bundle $N_{Z/X}$ of $Z$ in $X$ and $h\in H_{\bar{\partial}}^{1,1}(E)$ is the Dolbeault class of a first Chern form of the universal line bundle  $\mathcal{O}_E(-1)$ on $E$.
\end{thm}

This theorem is well-known for compact K\"{a}hler manifolds  by the blow-up formula for singular cohomology (\cite[Theorem 7.31]{Vo}) and Hodge decompositions.
With the same reasons, it also holds on the ones in the Fujiki class $\mathcal{C}$ (cf. \cite{RYY}).
Our theorem is applicable to noncompact  complex manifolds and the ones without Hodge decompositions.
The isomorphism of the type (\ref{b-u-m}) also exists in various homology theories, for instance, Chow group \cite{Ma,Vo}, singular cohomology \cite{GH,L,Vo}, Lawson homology \cite{H,N}, Deligne cohomology \cite{B-V,Gri} and Morse-Novikov cohomology \cite{YZ,M1}.

A gluing principle is given in Section 2 and is used  to prove Theorem \ref{L-H} in Section 3.
In Section 4, we first prove a self-intersection formula on Dolbeault cohomology for a special case and then prove Theorem \ref{main} by  the gluing principle. At last, the two theorems are applied to study strongly $q$-complete manifolds and the $\partial\bar{\partial}$-lemma, see Section 5.

\begin{rem}
After finishing the early version \cite{M2} of the present paper,  we made some applications in subsequent works \cite{M3,M4,M5,M6}. It is worth to notice that, Rao, S. et al. \cite{RYY2} extended the blow-up formula to bundle-valued Dolbeault cohomology on compact complex manifold and we  removed the assumption of the compactness in their results by Theorem \ref{main} and the gluing principle (the following Lemma \ref{glued}) in \cite{M4}.
\end{rem}

%==================================
\subsection*{Acknowledgements}
 I would like to express my gratitude to Prof. Fu, Jixiang  for his encouragement and support and thank School of Mathematical Science of Fudan University for the hospitality during my visit. Especially, I would like to thank Prof. Rao, Sheng  and Dr. Yang, Xiang-Dong  for many useful discussions and sending me the newest versions of their articles \cite{RYY,RYY2,YY}. I would like to thank the referee for valuable suggestions and
careful reading of my manuscript.
%==================================

\section{A gluing principle}
A classical technique in algebraic topology  is to  first consider the local cases and then extend them to the global cases by Mayer-Vietoris sequences. We refine it as a general principle and call it the \emph{gluing principle}.
\begin{lem}[{\cite[Lemma 3.7]{M1}}]\label{glued}
Let $X$ be a smooth manifold and denote by $\mathcal{P}(U)$ a statement on  any open subset $U\subseteq X$. Assume that $\mathcal{P}$ satisfies  conditions:

$(i)$ \emph{local condition}\emph{:} There exists a basis $\mathcal{U}$ of the topology of $X$, such that $\mathcal{P}(\bigcap\limits_{i=1}^lU_i)$ holds for any finite $U_1$, $\ldots$, $U_l$ $\in$ $\mathcal{U}$.

$(ii)$ \emph{disjoint condition}\emph{:} Let $\{U_n|n\in\mathbb{N}^+\}$ be a collection of disjoint open subsets of $X$. If $\mathcal{P}(U_n)$ hold for all $n\in\mathbb{N}^+$, $\mathcal{P}(\bigcup\limits_{n=1}^\infty U_n)$ holds.

$(iii)$ \emph{Mayer-Vietoris condition}\emph{:} For open subsets $U$, $V$ of $X$, if $\mathcal{P}(U)$, $\mathcal{P}(V)$ and $\mathcal{P}(U\cap V)$ hold, then $\mathcal{P}(U\cup V)$ holds.\\
Then $\mathcal{P}(X)$ holds.
\end{lem}
\begin{proof}
We first prove the following statement:

$(\textbf{*})$ For open subsets  $U_1$, $\ldots$, $U_r$ of $X$, if $\mathcal{P}(\bigcap\limits_{j=1}^kU_{i_j})$ holds for any $1\leq i_1<\ldots<i_k\leq r$, then $\mathcal{P}(\bigcup\limits_{i=1}^r U_i)$ holds.

For $r=1$, it holds obviously. Assume that $(\textbf{*})$ holds for $r$. For $r+1$, set $U'_1=U_1$, $\ldots$, $U'_{r-1}=U_{r-1}$, $U'_r=U_r\cup U_{r+1}$. Then $\mathcal{P}(\bigcap\limits_{j=1}^kU'_{i_j})$ holds for any $1\leq i_1<\ldots<i_k\leq r-1$. Moreover, by the Mayer-Vietoris condition, $\mathcal{P}(\bigcap\limits_{j=1}^{k-1}U'_{i_j}\cap U'_r)$ also holds for any $1\leq i_1<\ldots<i_{k-1}\leq r-1$, since $\mathcal{P}(\bigcap\limits_{j=1}^{k-1}U_{i_j}\cap U_r)$, $\mathcal{P}(\bigcap\limits_{j=1}^{k-1}U_{i_j}\cap U_{r+1})$ and  $\mathcal{P}(\bigcap\limits_{j=1}^{k-1}U_{i_j}\cap U_r\cap U_{r+1})$ all hold. By the inductive hypothesis, $\mathcal{P}(\bigcup\limits_{i=1}^{r+1} U_i)=\mathcal{P}(\bigcup\limits_{i=1}^{r} U'_i)$ holds. We proved $(\textbf{*})$.

Let $\mathcal{U}_\mathfrak{f}$ be the collection of open sets  which are the finite unions of open sets in  $\mathcal{U}$. We have a claim as follows:

$(\textbf{**})$ $\mathcal{P}(V)$ holds for any finite intersection $V$ of open sets in $\mathcal{U}_\mathfrak{f}$.

Assume that $V=\bigcap\limits_{i=1}^s U_i$, where $U_i=\bigcup\limits_{j=1}^{r_i}U_{ij}$ and $U_{ij}\in\mathcal{U}$. Then $V=\bigcup\limits_{J\in\Lambda}U_J$, where $\Lambda=\{J=(j_1,\ldots,j_s)|1\leq j_1\leq r_1,\ldots,1\leq j_s\leq r_s\}$ and $U_J=U_{1j_1}\cap...\cap U_{sj_s}$. For any $J_1$, $\ldots$, $J_t\in\Lambda$, $\mathcal{P}(U_{J_1}\cap \ldots\cap U_{J_t})$ holds by the local condition. Hence $\mathcal{P}(V)=\mathcal{P}(\bigcup\limits_{J\in\Lambda}U_J)$ holds by $(\textbf{*})$.

By \cite[p. 16, Proposition II]{GHV},  $X=V_1\cup\ldots\cup V_l$, where $V_i$ is a countable disjoint union of open sets in  $\mathcal{U}_\mathfrak{f}$. Obviously, for any $1\leq i_1<\ldots<i_k\leq l$, $\bigcap\limits_{j=1}^kV_{i_j}$ is a countable disjoint union of  finite intersections of open sets in $\mathcal{U}_\mathfrak{f}$. By the disjoint condition and $(\textbf{**})$,  $\mathcal{P}(\bigcap\limits_{j=1}^kV_{i_j})$ holds. So $\mathcal{P}(X)$ holds by $(\textbf{*})$.
\end{proof}

\section{A proof of Theorem \ref{L-H}}
For a complex manifold $X$, denote by $\mathcal{A}_X^{p,q}$ (resp. $\mathcal{D}_X^{\prime p,q}$, $\Omega_X^p$) the sheaf of germs of smooth $(p,q)$-forms (resp. $(p,q)$-currents, holomorphic $p$-forms) on $X$ and denote by $\mathcal{A}^{p,q}(X)$ (resp. $\mathcal{A}_c^{p,q}(X)$, $\mathcal{D}^{\prime p,q}(X)$) the space of smooth $(p,q)$-forms (resp. smooth $(p,q)$-forms with compact supports, $(p,q)$-currents) on $X$.

For bigraded vector spaces $K^{*,*}$ and $L^{*,*}$ over $\mathbb{C}$, the associated bigraded space $K^{*,*}\otimes_{\mathbb{C}}L^{*,*}$ over $\mathbb{C}$ is defined as
\begin{displaymath}
(K^{*,*}\otimes_{\mathbb{C}}L^{*,*})^{p,q}=\bigoplus\limits_{\substack{k+l=p\\r+s=q}}K^{k,r}\otimes_{\mathbb{C}} L^{l,s}
\end{displaymath}
for any $p$, $q$.

For coherent analytic sheaves $\mathcal{F}$ and $\mathcal{G}$  on complex manifolds $X$ and $Y$ respectively,  the (\emph{analytic}) \emph{external tensor product} of $\mathcal{F}$ and $\mathcal{G}$ is defined as
\begin{displaymath}
\mathcal{F}\boxtimes\mathcal{G}=pr_1^*\mathcal{F}\otimes_{\mathcal{O}_{X\times Y}}pr_2^*\mathcal{G},
\end{displaymath}
where $pr_1$ and $pr_2$ are projections from $X\times Y$ onto $X$, $Y$, respectively.  Clearly, $\Omega_{X\times Y}^p=\bigoplus\limits_{r+s=p}\Omega_{X}^r\boxtimes\Omega_{Y}^s$. Assume that $Y$  is \emph{compact}. By \cite[IX, (5.23) (5.24)]{D}, we have an isomorphism
\begin{equation}\label{Kun}
pr_1^*(\bullet)\cup pr_2^*(\bullet):H_{\bar{\partial}}^{*,*}(X)\otimes_{\mathbb{C}} H_{\bar{\partial}}^{*,*}(Y)\tilde{\rightarrow} H_{\bar{\partial}}^{*,*}(X\times Y)
\end{equation}
for any $p$, $q$.
It is the \emph{K\"{u}nneth formula} for the Dolbeault cohomology.

Combining the K\"{u}nneth formula and the gluing principle, we give a proof of Theorem \ref{L-H}.
\begin{proof}
Let $t_i$ be a $\bar{\partial}$-closed form of degree $(u_i,v_i)$ in $\mathcal{A}^{*,*}(E)$, such that $e_i=[t_i]$ for $1\leq i\leq r$. For any open set $U$ in $X$, set
\begin{displaymath}
B^{p,q}(U)=\bigoplus_{i=1}^{r}\mathcal{A}^{p-u_i,q-v_i}(U)
\end{displaymath}
and $\bar{\partial}_B=\bigoplus\limits_{i=1}^r\bar{\partial}: B^{p,q}(U)\rightarrow B^{p,q+1}(U)$. For any $p$, $(B^{p,\bullet}(U),\bar{\partial}_B)$ is a complex, whose cohomology is
\begin{displaymath}
D^{p,q}(U)=\bigoplus_{i=1}^{r}H_{\bar{\partial}}^{p-u_i,q-v_i}(U).
\end{displaymath}
Clearly, the morphism
\begin{displaymath}
\psi^{p}_U=\sum\limits_{i=1}^r\pi^*(\bullet)\wedge t_i:(B^{p,\bullet}(U),\bar{\partial}_B)\rightarrow (C^{p,\bullet}(U),\bar{\partial}_C):=(\mathcal{A}^{p,\bullet}(E_U),\bar{\partial})
\end{displaymath}
of complexes induces a morphism
\begin{displaymath}
\Psi^{p,q}_U=\sum\limits_{i=1}^r\pi^*(\bullet)\cup e_i:D^{p,q}(U)\rightarrow E^{p,q}(U):=H_{\bar{\partial}}^{p,q}(E_U).
\end{displaymath}
Denoted by $\mathcal{P}(U)$ the statement that $\Psi^{p,q}_U$ are isomorphisms for all $p$, $q$.  Our goal is to prove that $\mathcal{P}(X)$ holds by the gluing principle. Clearly, $\mathcal{P}$  satisfies the disjoint condition.

Given $p$, for any open subsets $U$, $V$ in $X$, there is a commutative diagram
\begin{displaymath}
\xymatrix{
0\ar[r]&B^{p,\bullet}(U\cup V)\ar[d]^{\psi^{p}_{U\cup V}}\quad \ar[r]^{(\rho_U^{U\cup V},\rho_V^{U\cup V})\quad\quad} & \quad B^{p,\bullet}(U)\oplus B^{p,\bullet}(V)\ar[d]^{(\psi^{p}_U,\psi^{p}_V)}\quad\ar[r]^{\quad\quad\rho_{U\cap V}^U-\rho_{U\cap V}^V}& \quad B^{p,\bullet}(U\cap V) \ar[d]^{\psi^{p}_{U\cap V}}\ar[r]& 0\\
0\ar[r]&C^{p,\bullet}(U\cup V)     \quad  \ar[r]^{(j_U^{U\cup V},j_V^{U\cup V})\quad\quad}& \quad C^{p,\bullet}(U)\oplus C^{p,\bullet}(V)    \quad\ar[r]^{\quad\quad j_{U\cap V}^U-j_{U\cap V}^ V} & \quad C^{p,\bullet}(U\cap V)     \ar[r]& 0}
\end{displaymath}
of  complexes, where $\rho$, $j$ are restrictions and the differentials of complexes in the top, bottom rows are all $\bar{\partial}_B$, $\bar{\partial}_C$, respectively. The two rows are both exact sequences of complexes. Therefore, we have a commutative diagram
\begin{displaymath}
\tiny{\xymatrix{
    \cdots\ar[r]&D^{p,q-1}(U\cap V)\ar[d]^{\Psi^{p,q-1}_{U\cap V}}\ar[r]&D^{p,q}(U\cup V) \ar[d]^{\Psi^{p,q}_{U\cup V}} \ar[r]& D^{p,q}(U)\oplus D^{p,q}(V) \ar[d]^{(\Psi^{p,q}_U,\Psi^{p,q}_V)}\ar[r]&  D^{p,q}(U\cap V)\ar[d]^{\Psi^{p,q}_{U\cap V}}\ar[r]&\cdots\\
 \cdots\ar[r]&E^{p,q-1}(U\cap V)    \ar[r] & E^{p,q}(U\cup V)\ar[r]&E^{p,q}(U)\oplus E^{p,q}(V)       \ar[r]& E^{p,q}(U\cup V)   \ar[r]&\cdots}}
\end{displaymath}
of long exact sequences. If $\Psi^{p,q}_U$, $\Psi^{p,q}_V$ and $\Psi^{p,q}_{U\cap V}$ are isomorphisms for all $p$, $q$, then so are $\Psi^{p,q}_{U\cup V}$ for all $p$, $q$ by the five-lemma. Hence $\mathcal{P}$  satisfies the Mayer-Vietoris condition.

To check the local condition, we first verify the following claim:

$(\lozenge)$ If the open set $U$ of $X$ satisfies that $E_U=\pi^{-1}(U)$ is holomorphically trivial, then $\mathcal{P}(U)$ holds.

Suppose that $F$ is the general fiber of $E$  and $\varphi_U: U\times F\rightarrow E_U$ is a holomorphic trivialization. Let $pr_1$ and $pr_2$ be projections from $U\times F$ to $U$ and $F$ respectively, which satisfy $\pi\circ \varphi_U=pr_1$.
Given a point $o\in U$, set $j_o:F\rightarrow U\times F$ as $f\mapsto (o,f)$.
Clearly, $pr_2\circ j_o=id_F$ and $i_o:=\varphi_U\circ j_o$ is the embedding $F\hookrightarrow E_U$ of the fiber $E_o\cong F$ over $o$ into $E_U$.
Set $e_i^{\prime}=(\varphi_U^{-1})^*pr_2^*i_o^*e_i\in H^{u_i,v_i}_{\bar{\partial}}(E_U)$, $1\leq i\leq r$.
Then $i_o^*e_i^{\prime}=i_o^*e_i$ for any $i$.
Since $i_o^*e_1$, $\ldots$, $i_o^*e_r$ is linearly independent,  mapping $e_i$ to $e'_i$ for $1\leq i\leq r$ give an isomorphism $\textrm{span}_{\mathbb{C}}\{e_1, \ldots, e_r\}\tilde{\rightarrow}\textrm{span}_{\mathbb{C}}\{e_1^{\prime}, \ldots, e_r^{\prime}\}$.
For any $p$, $q$, we have a commutative diagram

\begin{equation}\label{com}
\xymatrix{
 (H_{\bar{\partial}}^{*,*}(U)\otimes_{\mathbb{C}} \textrm{span}_{\mathbb{C}}\{e_1, \ldots, e_r\})^{p,q}\ar[d]^{\cong}  \ar[dr]^{id\otimes i_o^*}&\\
 (H_{\bar{\partial}}^{*,*}(U)\otimes_{\mathbb{C}} \textrm{span}_{\mathbb{C}}\{e_1^{\prime}, \ldots, e_r^{\prime}\})^{p,q}\ar[d]^{\pi^*(\bullet)\cup \bullet} \ar[r]^{\qquad id\otimes i_o^*}& (H_{\bar{\partial}}^{*,*}(U)\otimes_{\mathbb{C}} H_{\bar{\partial}}^{*,*}(F))^{p,q} \ar[d]^{pr_1^*(\bullet)\cup pr_2^*(\bullet)}\\
 H_{\bar{\partial}}^{p,q}(E_U) \ar[r]^{\varphi_U^*}&  H_{\bar{\partial}}^{p,q}(U\times F), }
\end{equation}
where $\textrm{span}_{\mathbb{C}}\{e_1, \ldots, e_r\}$ and $\textrm{span}_{\mathbb{C}}\{e_1^{\prime}, \ldots, e_r^{\prime}\}$ are viewed as bigraded subspaces of $H^{*,*}_{\bar{\partial}}(E_U)$.
By the assumption, the restriction of $i_o^*$ to $\textrm{span}_{\mathbb{C}}\{e_1, \ldots, e_r\}$, hence to $\textrm{span}_{\mathbb{C}}\{e_1^{\prime}, \ldots, e_r^{\prime}\}$, is an isomorphism. Notice that $F$ is compact. By the K\"{u}nneth formula (\ref{Kun}), $pr_1^*(\bullet)\cup pr_2^*(\bullet)$ is isomorphic,  so is $\pi^*(\bullet)\cup \bullet$ in (\ref{com}).
Mapping $(\alpha_1,\ldots,\alpha_r)$ to $\sum\limits_{i=1}^r\alpha_i\otimes e_i$ gives a morphism
\begin{equation}\label{id}
\sum\limits_{i=1}^r\bullet\otimes e_i:\bigoplus\limits_{i=1}^{r}H_{\bar{\partial}}^{p-u_i,q-v_i}(U)\rightarrow (H_{\bar{\partial}}^{*,*}(U)\otimes_{\mathbb{C}} \textrm{span}_{\mathbb{C}}\{e_1, \ldots, e_r\})^{p,q},
\end{equation}
which is clearly isomorphic. Then $\Psi^{p,q}_U$ is the composition of (\ref{id}) and the  two vertical maps in the first column of  (\ref{com}), hence an isomorphism.
We proved $(\lozenge)$.
Let $\mathcal{U}$ be a basis of the topology of $X$ such that $E_U$ is holomorphically trivial for any $U\in\mathcal{U}$. For $U_1$, $\ldots$, $U_l\in\mathcal{U}$,  $E_{U_1\cap\ldots\cap U_l}$ is holomorphically trivial, then $(\lozenge)$ asserts that $\mathcal{P}(\bigcap\limits_{i=1}^lU_i)$ is an isomorphism. Hence $\mathcal{P}$ satisfies the local condition.

We complete the proof.
\end{proof}
\begin{rem}
%Let $\pi:E\rightarrow X$ be a holomorphic fiber bundle with compact fibers and $j:F\rightarrow E$ be the natural inclusion of the general fiber.
Cordero, L. et al. \cite[Lemma 18]{CFGU} first give a Leray-Hirsch theorem by a spectral sequence introduced by Borel, A..
On their theorem, we want to point the follows:

$(1)$ They required that the Dolbeault cohomology algebra $H_{\bar{\partial}}^{*,*}(F)$ of the general fiber $F$ need to have a  transgressive algebraic basis.
In fact, they only used the existence of a bigraded vector space basis consisting of trangressive elements (for details, see \cite[Lemma 18]{CFGU}).
This assumption  is equivalent to say that, for any $\bar{\partial}$-closed form $a\in \mathcal{A}^{*,*}(F)$, there exist $\tilde{a}\in \mathcal{A}^{*,*}(E)$ and a $\bar{\partial}$-closed form $b\in \mathcal{A}^{*,*}(X)$ such that $j^*\tilde{a}=a$ and $\bar{\partial}\tilde{a}=\pi^*b$, where $j:F\rightarrow E$ be the inclusion of the general fiber.

$(2)$ Their assumption  is \emph{weaker} than ours in Theorem 1.1.
This point was also mentioned in \cite[Page 8]{RYY2}.
Actually,  for any $\bar{\partial}$-closed form $a\in \mathcal{A}^{*,*}(F)$, our assumptions imply that  %for any $\alpha\in H^{*,*}(F)$, there exists $\tilde{\alpha}$  $\in H^{*,*}(E)$ such that $j^*\tilde{\alpha}=\alpha$.
there exist a $\bar{\partial}$-closed form $\tilde{a}_1\in \mathcal{A}^{*,*}(E)$ and $c\in \mathcal{A}^{*,*}(F)$ such that $j^*\tilde{a}_1=a+\bar{\partial}c$.
Since $j^*:\mathcal{A}^{*,*}(E)\rightarrow \mathcal{A}^{*,*}(F)$ is surjective, $j^*\tilde{c}=c$ for some $\tilde{c}\in \mathcal{A}^{*,*}(E)$.
Set $\tilde{a}=\tilde{a}_1-\bar{\partial}\tilde{c}$.
Then $j^*\tilde{a}=a$ and $\bar{\partial}\tilde{a}=0$.
\end{rem}

%{\cite[Proposition 3.3]{RYY}}{\cite[Proposition 2]{ASTT}}{\cite[Proposition 4]{St2}}
\begin{cor}[{\cite{RYY,ASTT,St2}}]\label{poj-D}
Suppose that $\pi:\mathbb{P}(E)\rightarrow X$ is the projectivization of a holomorphic vector bundle $E$ on a connected complex manifold $X$. Let $h\in H_{\bar{\partial}}^{1,1}({\mathbb{P}(E)})$ be the Dolbeault class of a first Chern form of the universal line bundle $\mathcal{O}_{\mathbb{P}(E)}(-1)$ on ${\mathbb{P}(E)}$. Then
\begin{displaymath}
\sum\limits_{i=0}^{r-1}\pi^*(\bullet)\cup h^i:\bigoplus\limits_{i=0}^{r-1}H_{\bar{\partial}}^{p-i,q-i}(X)\rightarrow H_{\bar{\partial}}^{p,q}(\mathbb{P}(E)),
\end{displaymath}
is an isomorphism for any $p$, $q$,
where $\emph{rank}_{\mathbb{C}}E=r$.
\end{cor}
\begin{proof}
For every $x\in X$, $1$, $h$, $\ldots$, $h^{r-1}$ restricted to the fibre $\pi^{-1}(x)=\mathbb{P}(E_x)$ freely linearly generate $H_{\bar{\partial}}^{*,*}(\mathbb{P}(E_x))$. By Theorem \ref{L-H}, we immediately get the corollary.
\end{proof}
%\begin{rem}
%Cordero, L. A. gave another version of the Leray-Hirsch Theorem by a spectral sequence introduced by Borel, A., where $H_{\bar{\partial}}^{*,*}(F)$ need to have a  \emph{transgressive algebraic}  basis, refer to \cite[Lemma 18]{CFGU}. Using this theorem, Rao, S., Yang, S. and Yang, X.-D. \cite{RYY} proved Corollary \ref{poj-D} for the compact cases. %Actually, their proof  also holds for possibly noncompact complex manifold $X$, since Cordero's Hirsch lemma holds for any  base space.
%Another two proofs of Corollary  \ref{poj-D} were given by Angella, D. et al. in \cite{ASTT} on $\partial\bar{\partial}$-manifolds and by Stelzig, J. in \cite{S1} on compact complex manifolds.
%\end{rem}
\begin{rem}
With the similar proof, we can get the flag bundle formula for Dolbeault cohomology, see the proof of Corollary \ref{flagbun}.
\end{rem}

\begin{rem}\label{Hopf}
There exists holomorphic fiber bundles which do not satisfy the condition in Theorem \ref{L-H}.
For example, let $W=\mathbb{C}^2-\{(0,0)\}$ and $G$ the infinite cyclic group which is generated by the automorphism $(z_1,z_2)\mapsto (a_1z_1,a_2z_2)$ with $0< |a_1|$, $|a_2|<1$ and $a_1^k=a_2^l$ for some $k$, $l\in\mathbb{Z}$.
Set $\alpha=(a_1,a_2)$ and denote by $H_\alpha$ the \emph{Hopf surface} $W/G$.
Then $H_\alpha$ is an elliptic fiber bundle over $\mathbb{C}P^1$ (\cite[V, Proposition (18.2)]{BHPV}).
The $(1,0)$-Hodge numbers of  $H_\alpha$ and any fiber of $H_\alpha$ are $0$ (\cite[V, Proposition (18.1)]{BHPV}) and $1$ respectively, so the fiber bundle $H_\alpha$ over $\mathbb{C}P^1$ does not satisfy the condition in Theorem \ref{L-H}.
\end{rem}

\section{A proof of Theorem \ref{main}}
Now, we recall two sheaves defined  in the third version of  \cite{RYY} and study their properties by a little different approach. Following the ideas in \cite{RYY}, we fill some details of their proof with results in \cite{Gro,GN,St2}.

Let $X$ be a complex manifold and $i:Z\rightarrow X$ the inclusion of a complex submanifold $Z$ into $X$. For any $p$, $q$, the pullpacks give an epimorphism $\mathcal{A}_X^{p,q}\rightarrow i_*\mathcal{A}_Z^{p,q}$ of sheaves, which can be checked locally (or see \cite[Lemma 3.9]{RYY}). Denote
\begin{displaymath}
\mathcal{G}_{X,Z}^{p,q}=\textrm{ker}(\mathcal{A}_X^{p,q}\rightarrow i_*\mathcal{A}_Z^{p,q})
\end{displaymath}
and by $\mathcal{C}^{\infty}_X$ the sheaf of germs of complex-valued smooth functions on $X$.  Clearly, $\mathcal{G}_{X,Z}^{p,q}$ is a sheaf of $\mathcal{C}^{\infty}_X$-module, hence a soft sheaf. For any $p$, we have a short exact sequence
\begin{equation}\label{short0}
0\rightarrow(\mathcal{G}_{X,Z}^{p,\bullet},\bar{\partial})\rightarrow(\mathcal{A}_X^{p,\bullet},\bar{\partial})\rightarrow (i_*\mathcal{A}_Z^{p,\bullet},\bar{\partial})\rightarrow 0
\end{equation}
of complexes of sheaves, which induces a long exact sequence
\begin{equation}\label{long}
\xymatrix{
    \cdots\ar[r]&\mathcal{H}^{q-1}(i_*\mathcal{A}_Z^{p,\bullet},\bar{\partial})\ar[r]&\mathcal{H}^{q}(\mathcal{G}_{X,Z}^{p,\bullet},\bar{\partial})  \ar[r]& \mathcal{H}^{q}(\mathcal{A}_X^{p,\bullet},\bar{\partial}) \ar[r]&  \mathcal{H}^{q}(i_*\mathcal{A}_Z^{p,\bullet},\bar{\partial})\ar[r]&\cdots }
\end{equation}
of sheaves, where $\mathcal{H}^q(\mathcal{R}^\bullet)$ is the $q$-th cohomology sheaf of the complex $\mathcal{R}^\bullet$ of sheaves. As we know, $\mathcal{H}^{q}(\mathcal{A}_X^{p,\bullet},\bar{\partial})$ is $\Omega_X^p$ at $q=0$ and $0$ at $q>0$. Define
\begin{displaymath}
\mathcal{F}_{X,Z}^p=\textrm{ker}(\bar{\partial}:\mathcal{G}_{X,Z}^{p,0}\rightarrow\mathcal{G}_{X,Z}^{p,1}).
\end{displaymath}
Since $i_*$ is an exact functor, by $($\ref{long}$)$,
\begin{displaymath}
0\rightarrow\mathcal{F}_{X,Z}^p\rightarrow\Omega_X^p\rightarrow i_*\Omega_Z^p\rightarrow \mathcal{H}^1(\mathcal{G}^{p,\bullet},\bar{\partial})\rightarrow 0
\end{displaymath}
is exact and $\mathcal{H}^q(\mathcal{G}^{p,\bullet},\bar{\partial})=0$ for $q\geq2$. It is easily checked that the pullbacks induce an epimorphism $\Omega_X^p\rightarrow i_*\Omega_Z^p$. So
\begin{equation}\label{short}
0\rightarrow\mathcal{F}_{X,Z}^p\rightarrow\Omega_X^p\rightarrow i_*\Omega_Z^p\rightarrow 0
\end{equation}
is an exact sequence of sheaves and
\begin{displaymath}
\xymatrix{
0\ar[r] &\mathcal{F}_{X,Z}^p\ar[r]^{i} &\mathcal{G}_{X,Z}^{p,0}\ar[r]^{\bar{\partial}} &\mathcal{G}_{X,Z}^{p,1}\ar[r]^{\bar{\partial}}&\cdots\ar[r]^{\bar{\partial}}&\mathcal{G}_{X,Z}^{p,n}\ar[r]&0
}
\end{displaymath}
is a resolution of soft sheaves of $\mathcal{F}_{X,Z}^p$, for any $p$.

Now, let $\pi:\widetilde{X}_Z\rightarrow X$ be the blow-up of a connected (not necessary compact) complex manifold $X$ along a connected complex submanifold $Z$. We know $\pi|_E:E=\pi^{-1}(Z)\rightarrow Z$ is the projectivization $E=\mathbb{P}(N_{Z/X})$ of the normal bundle $N_{Z/X}$ of $Z$ in $X$. Let $i_Z:Z\rightarrow X$
and $i_E:E\rightarrow \widetilde{X}_Z$ be inclusions.

The following lemma was first proved for smooth schemes over a field in \cite{Gro, GN}. Stelzig, J. indicated that it also holds for complex manifolds and the proof can be copied verbatim. Recently, Rao, S. et al. (\cite{RYY2}) gave all details of the proof of this lemma on compact bases in the holomorphic category by the same steps in  \cite{GN} and moreover, they actually proved it for the bundle-valued case. Following their steps (\cite{GN, RYY2}), we can easily extend it to possibly noncompact manifolds with minor modifications.

\begin{lem}\label{1}
For any $p\geq0$, the following morphisms are isomorphisms:

$(1)$ $\pi^*:\Omega_X^p\tilde{\rightarrow}\pi_*\Omega_{\widetilde{X}_Z}^p$,

$(2)$  $(\pi|_E)^*:\Omega_Z^p\tilde{\rightarrow}(\pi|_E)_*\Omega_{E}^p$,

$(3)$ $i_E^*:R^q\pi_*\Omega_{\widetilde{X}_Z}^p\tilde{\rightarrow}i_{Z*}R^q(\pi|_E)_*\Omega_{E}^p$, for $q\geq1$.
\end{lem}
\begin{proof}
The proof of \cite[Lemma 4.1 (i)(ii)]{RYY2} also hold for possibly noncompact manifolds, so we obtain $(1)$ and $(2)$. For $(3)$, let $\mathcal{O}_{\widetilde{X}_Z}(1)$ be the ideal sheaf $\mathcal{O}_{\widetilde{X}_Z}(-E)$ associated to the divisor $-E$ on $\widetilde{X}_Z$. %Then $\mathcal{O}_{\widetilde{X}_Z}(1)|_E=\mathcal{O}_{\widetilde{X}_Z}(-E)|_E=\mathcal{O}_{\mathbb{P}(\mathcal{N}_{Z/X})}(1)$. $\mathcal{O}_{\widetilde{X}_Z}(1)$  is ample with respect to $\pi$.
Choose an open covering $\mathcal{U}$ of $X$ such that the closure $\overline{U}$ is compact for any $U\in\mathcal{U}$. Since the blow-up $\pi$ is a projective morphism (\cite[p. 290, Remark 2.1 (0)]{GPR}), by Grauert-Remmert theorem (\cite[IV, Theorem 2.1 (B))]{BS}, there exists an integer $m_U$, such that
\begin{displaymath}
R^q\pi_*(\Omega^p_{\widetilde{X}_Z}\otimes_{\mathcal{O}_{\widetilde{X}_Z}}\mathcal{O}_{\widetilde{X}_Z}(l))=0
\end{displaymath}
on $\overline{U}$ for $q\geq1$ and $l\geq m_U$, where $m_U$ only depends on $\bar{U}$ and $\Omega^p_{\widetilde{X}_Z}$. Following the steps of  \cite[Proposition 3.3]{Gro} or \cite[Lemma 4.1]{RYY2} on the base $U$, we get an isomorphism
\begin{displaymath}
i_{E\cap \widetilde{U}}^*:R^q\pi_*\Omega_{\widetilde{U}_{Z\cap U}}^p\tilde{\rightarrow}i_{Z\cap U*}R^q(\pi|_{E\cap \widetilde{U}})_*\Omega_{E\cap \widetilde{U}}^p,
\end{displaymath}
for $q\geq1$, which is just the restriction of $i_E^*$ to $U$. Since $X$ is covered by $\mathcal{U}$,  $i_E^*$ is an isomorphism on $X$.
\end{proof}

\begin{lem}\label{2}
For any $p\geq0$,
\begin{displaymath}
R^q\pi_*\mathcal{F}^p_{\widetilde{X}_Z,E}=\left\{
 \begin{array}{ll}
\mathcal{F}^p_{X,Z},&~q=0\\
 &\\
 0,&~q\geq1.
 \end{array}
 \right.
\end{displaymath}
\end{lem}
\begin{proof}
For the short exact sequence $($\ref{short}$)$ of the pair $(\widetilde{X}_Z,E)$, we have a long exact sequence
\begin{equation}\label{long2}
\xymatrix{
0\ar[r] &\pi_*\mathcal{F}_{\widetilde{X}_Z,E}^p\ar[r] &\pi_*\Omega_{\widetilde{X}_Z}^p\ar[r] &\pi_*i_{E*}\Omega_{E}^p\ar[r]&R^1\pi_*\mathcal{F}_{\widetilde{X}_Z,E}^p\\
\ar[r]&R^1\pi_*\Omega_{\widetilde{X}_Z}^p\ar[r]&R^1\pi_*(i_{E*}\Omega_{E}^p)\ar[r]&\cdots\ar[r]&R^q\pi_*\mathcal{F}_{\widetilde{X}_Z,E}^p\\
\ar[r]&R^q\pi_*\Omega_{\widetilde{X}_Z}^p\ar[r]&R^q\pi_*(i_{E*}\Omega_{E}^p)\ar[r]&\cdots.
}
\end{equation}
Let $\Omega_{E}^p\rightarrow\mathcal{I}^\bullet$ be an injective resolution of $\Omega_{E}^p$. Then $i_{E*}\Omega_{E}^p\rightarrow i_{E*}\mathcal{I}^\bullet$ is also an injective resolution. Since $i_{Z*}$ is an exact functor,
\begin{displaymath}
R^q\pi_*(i_{E*}\Omega_{E}^p)=\mathcal{H}^q(\pi_*(i_{E*}\mathcal{I}^\bullet))=\mathcal{H}^q(i_{Z*}(\pi|_E)_*\mathcal{I}^\bullet))=i_{Z*}R^q(\pi|_E)_*\Omega_{E}^p.
\end{displaymath}
By Lemma \ref{1} (3) and the sequence (\ref{long2}), we get $R^q\pi_*\mathcal{F}^p_{\widetilde{X}_Z,E}=0$ for $q\geq2$. Moreover, by Lemma \ref{1} (1) (2), $\pi_*\Omega_{\widetilde{X}_Z}^p\rightarrow\pi_*i_{E*}\Omega_{E}^p$ is just $\Omega_{X}^p\rightarrow i_{Z*}\Omega_{Z}^p$, which is epimorphic. So $\mathcal{F}^p_{X,Z}\cong\pi_*\mathcal{F}_{\widetilde{X}_Z,E}^p$ and $R^1\pi_*\mathcal{F}^p_{\widetilde{X}_Z,E}=0$. We complete the proof.
\end{proof}

\begin{rem}
The first part of this lemma was proved with a direct way in the third version of \cite{RYY}  (cf. Lemma 3.10) and the second part seems not to be treated there. After them, we first gave the above proof in the early version \cite{M2} of the present article. Recently, Rao, S. et al. \cite{RYY2} studied this lemma for the bundle-valued cases with our approach. Actually, their results can be obtained directly by ours and the projection formula.
\end{rem}

By the Leray spectral sequence and Lemma \ref{2}, $\pi^*$ induces an isomorhism
\begin{displaymath}
H^q(X,\mathcal{F}_{X,Z}^p)\cong H^q(X,\pi_*\mathcal{F}_{\widetilde{X}_Z,E}^p)\cong H^q(\widetilde{X}_Z,\mathcal{F}_{\widetilde{X}_Z,E}^p).
\end{displaymath}
For a fixed $p$.  From the sequence (\ref{short0}), we get a commutative diagram
\begin{displaymath}
\small{\xymatrix{
 0\ar[r]&\Gamma(X,\mathcal{G}_{X,Z}^{p,\bullet})\ar[d]^{\pi^*} \ar[r]& \Gamma(X,\mathcal{A}_{X}^{p,\bullet})\ar[d]^{\pi^*} \ar[r]^{i_Z^*}& \Gamma(Y,\mathcal{A}_{Z}^{p,\bullet}) \ar[d]^{(\pi|_E)^*}\ar[r]& 0\\
 0\ar[r]&\Gamma(\widetilde{X},\mathcal{G}_{\widetilde{X},E}^{p,\bullet})    \ar[r]^{}& \Gamma(\widetilde{X},\mathcal{A}_{\widetilde{X}}^{p,\bullet})  \ar[r]^{i_E^*} &  \Gamma(E,\mathcal{A}_{E}^{p,\bullet})    \ar[r]& 0}}
\end{displaymath}
of short exact sequences of complexes, since $\mathcal{G}_{X,Y}^{*,*}$ and $\mathcal{G}_{\widetilde{X},E}^{*,*}$ are soft sheaves. It induces a commutative diagram
\begin{displaymath}
\small{\xymatrix{
    \cdots\ar[r]&H^q(X,\mathcal{F}_{X,Z}^{p}) \ar[d]^{\cong} \ar[r]& H_{\overline{\partial}}^{p,q}(X) \ar[d]^{\pi^*}\ar[r]^{i_Z^*}&  H_{\overline{\partial}}^{p,q}(Z)\ar[d]^{(\pi|_E)^*}\ar[r]&H^{q+1}(X,\mathcal{F}_{X,Z}^{p})\ar[d]^{\cong}\ar[r]&\cdots\\
 \cdots \ar[r] & H^q(\widetilde{X}_Z,\mathcal{F}_{\widetilde{X}_Z,E}^{p})\ar[r]& H_{\overline{\partial}}^{p,q}(\widetilde{X})       \ar[r]^{i_E^*}& H_{\overline{\partial}}^{p,q}(E)     \ar[r] & H^{q+1}(\widetilde{X}_Z,\mathcal{F}_{\widetilde{X}_Z,E}^{p})\ar[r]&\cdots}}
\end{displaymath}
of long exact sequences, where $\pi^*$ and $(\pi|_E)^*$ are injective by the projection formula (\cite[VI, (12.8)]{D}) and Corollary \ref{poj-D}, respectively. By the snake lemma, $i_E^*$ induces an isomorphism
$\textrm{coker}\pi^*\tilde{\rightarrow}\textrm{coker}(\pi|_E)^*$. We get a commutative diagram
\begin{equation}\label{commutative1}
\xymatrix{
 0\ar[r]&H_{\bar{\partial}}^{p,q}(X)\ar[d]^{i_Z^*} \ar[r]^{\pi^*}& H_{\bar{\partial}}^{p,q}(\widetilde{X}_Z)\ar[d]^{i_E^*} \ar[r]& \textrm{coker}\pi^* \ar[d]^{\cong}\ar[r]& 0\\
 0\ar[r]&H_{\bar{\partial}}^{p,q}(Z)       \ar[r]^{(\pi|_E)^*}& H_{\bar{\partial}}^{p,q}(E)   \ar[r]^{} &  \textrm{coker} (\pi|_E)^*    \ar[r]& 0 }
\end{equation}
of short exact sequences for any $p$, $q$.

Denote by $\Theta(\mathcal{O}_E(-1))$ the curvature of the Chern connection of a hermitian metric of the universal bundle $\mathcal{O}_E(-1)$ over $E$. Then $\frac{i}{2\pi}\Theta(\mathcal{O}_E(-1))\in \mathcal{A}^{1,1}(E)$ is real and $\bar{\partial}$-closed.
Set $h=[\frac{i}{2\pi}\Theta(\mathcal{O}_E(-1))]$ the Dolbeault class of $\frac{i}{2\pi}\Theta(\mathcal{O}_E(-1))$ in $H_{\bar{\partial}}^{1,1}(E)$. We first consider a special case of  Theorem \ref{main}.
\begin{lem}[Self-intersection formula]\label{key}
Suppose that $Z$ is a Stein manifold. Then
\begin{displaymath}
i_E^*i_{E*}\sigma=h\cup\sigma
\end{displaymath}
%$i_E^*i_{E*}\sigma=h\cup\sigma$,
for any $\sigma\in H_{\bar{\partial}}^{*,*}(E)$.
\end{lem}
\begin{proof}
By \cite[Theorem 3.3.3]{Fo}, we can choose a Stein neighborhood  $U$  of $Z$ such that there exists a holomorphic map $\tau: U\rightarrow Z$ satisfying $\tau\circ l_Z=id_Z$, where $l_Z:Z\rightarrow U$ is the inclusion.
Let $\alpha\in \mathcal{A}^{p,q}(E)$ be a representative of $\sigma$.  By $($\ref{commutative1}$)$, $l_E^*$ induces the isomorphism
\begin{displaymath}
H_{\bar{\partial}}^{p,q}(\widetilde{U})/\pi^*H_{\bar{\partial}}^{p,q}(U)\tilde{\rightarrow} H_{\bar{\partial}}^{p,q}(E)/(\pi|_E)^*H_{\bar{\partial}}^{p,q}(Z),
\end{displaymath}
where $l_E:E\rightarrow\widetilde{U}$ is the inclusion. Hence, there exist a form $\eta\in\mathcal{A}^{p,q-1}(E)$ and $\bar{\partial}$-closed forms $\gamma\in\mathcal{A}^{p,q}(\widetilde{U})$, $\delta\in\mathcal{A}^{p,q}(Z)$, such that $\alpha=l_E^*\gamma+(\pi|_E)^*\delta+\bar{\partial}\eta$.  So
\begin{displaymath}
\alpha=l_E^*\left(\gamma+(\pi|_{\tilde{U}})^*\tau^*\delta\right)+\bar{\partial}\eta.
\end{displaymath}

In $H_{dR}^2(\widetilde{U})$, $[l_{E*}(1)]=[E]$, where $[E]$ is the fundamental class of $E$ in $\widetilde{U}$. Since $[E]=c_1(\mathcal{O}_{\widetilde{U}}(E))=[\frac{i}{2\pi}\Theta(\mathcal{O}_{\widetilde{U}}(E))]\in H_{dR}^2(\widetilde{U})$, $l_{E*}(1)=\frac{i}{2\pi}\Theta(\mathcal{O}_{\widetilde{U}}(E))+dR$ for a real current $R$ of degree $1$ on $\widetilde{U}$. Set $R=S+\bar{S}$, where $S\in\mathcal{D}^{\prime0,1}(\widetilde{U})$. Thus $\bar{\partial}S=0$ by comparing the degrees. By \cite[Corollary 2.15]{U},  $H_{\bar{\partial}}^{0,1}(\widetilde{U})=H_{\bar{\partial}}^{0,1}(U)=0$, since  $U$ is Stein. So $S=\bar{\partial}S_1$ for some $S_1\in\mathcal{D}^{\prime0,0}(\widetilde{U})$. Set $T=\partial(\bar{S_1}-S_1)$. Then
\begin{displaymath}
l_{E*}(1)=\frac{i}{2\pi}\Theta(\mathcal{O}_{\widetilde{U}}(E))+\bar{\partial}T.
\end{displaymath}
By the projection formula of currents,
\begin{displaymath}
\begin{aligned}
l_{E*}\alpha=&l_{E*}(1)\wedge\left(\gamma+(\pi|_{\widetilde{U}})^*\tau^*\delta\right)+\bar{\partial}(l_{E*}\eta)\\
=&\frac{i}{2\pi}\Theta(\mathcal{O}_{\widetilde{U}}(E))\wedge(\gamma+(\pi|_{\widetilde{U}})^*\tau^*\delta)+\bar{\partial}Q,
\end{aligned}
\end{displaymath}
where $Q=l_{E*}\eta+T\wedge(\gamma+(\pi|_{\widetilde{U}})^*\tau^*\delta)\in\mathcal{D}^{\prime p+1,q}(\widetilde{U})$.

 Let $j:\widetilde{U}\rightarrow \tilde{X}_Z$ be the inclusion. Since $\textrm{supp}(l_{E*}\alpha)\subseteq E$, $j|_{\textrm{supp}(l_{E*}\alpha)}$ is proper, which implies that $j_*(l_{E*}\alpha)$ is defined well.  Clearly, $j_*(l_{E*}\alpha)=i_{E*}\alpha$ and $j^*j_*l_{E*}\alpha=l_{E*}\alpha$.  So
\begin{displaymath}
\begin{aligned}
i_E^*i_{E*}\sigma=&i_E^*[i_{E*}\alpha]=l_E^*[j^*j_*l_{E*}\alpha]\\
=&l_E^*[l_{E*}\alpha]\\
=&l_E^*[\frac{i}{2\pi}\Theta(\mathcal{O}_{\widetilde{U}}(E))]\cup[\alpha]\\
=&h\cup\sigma,
\end{aligned}
\end{displaymath}
where $[\beta]$ denotes the Dolbeault  class of  the $\bar{\partial}$-closed form $\beta$ and  the last equation used the fact that $\mathcal{O}_E(-1)=\mathcal{O}_{\widetilde{U}}(E)|_E$.
\end{proof}

\begin{prop}\label{special}
Theorem \ref{main} holds for the case that $Z$ is a Stein manifold.
\end{prop}
\begin{proof}
%By Proposition \ref{poj-D}, we have the isomorphism
%\begin{equation}\label{pr-bun-1}
%\sum\limits_{i=0}^{r-1}(\pi|_E)^*(\bullet)\cup h^i: H_{\bar{\partial}}^{p-i,q-i}(Z)\tilde{\rightarrow}  H_{\bar{\partial}}^{p,q}(E).
%\end{equation}
Suppose that $\pi^*\alpha^{p,q}+\sum\limits_{i=1}^{r-1}i_{E*}\left(h^{i-1}\cup(\pi|_E)^*\beta^{p-i,q-i}\right)=0$, where $\alpha^{p,q}\in H_{\bar{\partial}}^{p,q}(X)$ and $\beta^{p-i,q-i}\in H_{\bar{\partial}}^{p-i,q-i}(Z)$ for $1\leq i\leq r-1$. Pull it back  by $i_E^*$, we get
\begin{displaymath}
(\pi|_E)^*i_Z^*\alpha^{p,q}+\sum_{i=1}^{r-1}h^{i}\cup(\pi|_E)^*\beta^{p-i,q-i}=0
\end{displaymath}
by Lemma \ref{key}, where $i_Z:Z\rightarrow X$ is the inclusion.
By Corollary \ref{poj-D}, $\beta^{p-i,q-i}=0$ for every $i$. So $\pi^*\alpha^{p,q}=0$. By \cite[Theorem 3.1]{W}, $\pi^*$ is injective, so $\alpha^{p,q}=0$. The injectivity of the morphism $($\ref{b-u-m}$)$ is true.

For any $\gamma\in H_{\bar{\partial}}^{p,q}(\widetilde{X}_Z)$, by Corollary \ref{poj-D}, there exist $\beta^{p-i,q-i}\in H_{\bar{\partial}}^{p-i,q-i}(Z)$ for $0\leq i\leq r-1$, such that $i_E^*\gamma=\sum\limits_{i=0}^{r-1}h^i\cup(\pi|_E)^*\beta^{p-i,q-i}$. By Lemma \ref{key},
\begin{displaymath}
i_E^*\left[\gamma-\sum_{i=1}^{r-1}i_{E*}\left(h^{i-1}\cup(\pi|_E)^*\beta^{p-i,q-i}\right)\right]=(\pi|_E)^*\beta^{p,q},
\end{displaymath}
which is zero in $\textrm{coker}(\pi|_E)^*$. By the commutative diagram (\ref{commutative1}),
\begin{displaymath}
\gamma-\sum_{i=1}^{r-1}i_{E*}\left(h^{i-1}\cup(\pi|_E)^*\beta_{p-i,q-i}\right)=\pi^*\alpha^{p,q},
\end{displaymath}
for some $\alpha^{p,q}\in H_{\bar{\partial}}^{p,q}(X)$. So the morphism $($\ref{b-u-m}$)$ is surjective.

We complete the proof.
\end{proof}

By definitions, we easily get
\begin{lem}\label{com}
Let $Z$ be a closed complex submanifold and  $U$ an open complex submanifold of the complex manifold $X$. Consider the cartesian
\begin{displaymath}
\xymatrix{
  Z\cap U\ar[d]_{j'}  \ar[r]^{\quad i'}  & U \ar[d]^{j} \\
  Z \ar[r]^{i}  & X,}
\end{displaymath}
where $i$, $j$, $i'$ and $j'$ are inclusions. Then, $j'_*i'^*=i^*j_*$ on $\mathcal{A}_c^{**}(U)$. Dually, $i'_*j'^*=j^*i_*$ on $\mathcal{D}^{\prime**}(Z)$.
\end{lem}

After above preparations, we can prove Theorem \ref{main}.
\begin{proof}
For any $p$, $q$, set
\begin{displaymath}
\mathcal{F}^{p,q}=\mathcal{A}_X^{p,q}\oplus \bigoplus_{i=1}^{r-1}i_{Z*}\mathcal{A}_Z^{p-i,q-i}.
\end{displaymath}
Define $\bar{\partial}:\mathcal{F}^{p,*}\rightarrow\mathcal{F}^{p,*+1}$ as $(\alpha,\beta_0,\ldots,\beta_{r-2})\mapsto (\bar{\partial}\alpha,\bar{\partial}\beta_0,\ldots,\bar{\partial}\beta_{r-2})$. For any $p$, $(\mathcal{F}^{p,\bullet},\bar{\partial})$ is a complex of sheaves.
Let $t=\frac{i}{2\pi}\Theta(\mathcal{O}_E(-1))\in\mathcal{A}^{1,1}(E)$. For any open subset $U$ in $X$, define $\mathcal{F}^{p,q}(U)\rightarrow \mathcal{D}^{\prime p,q}(\widetilde{U})$ as
\begin{displaymath}
\varphi^{p,q}_U=\left\{
 \begin{array}{ll}
(\pi|_{\widetilde{U}})^*+\sum\limits_{i=1}^{r-1}i_{E\cap \widetilde{U}*}\circ (t^{i-1}|_{E\cap \widetilde{U}}\wedge)\circ (\pi|_{E\cap \widetilde{U}})^*,&~Z\cap U\neq\emptyset\\
 &\\
 (\pi|_{\widetilde{U}})^*,&~Z\cap U=\emptyset,
 \end{array}
 \right.
\end{displaymath}
where $i_{E\cap \widetilde{U}}:E\cap \widetilde{U}\rightarrow \widetilde{U}$ is the inclusion. Evidently, $\bar{\partial}\circ\varphi^{p,q}_U=\varphi^{p,q+1}_U\circ\bar{\partial}$. Thus $\varphi^{p,\bullet}_U$ induces a morphism
\begin{displaymath}
\Phi^{p,q}_U:H_{\bar{\partial}}^{p,q}(U)\oplus \bigoplus_{i=1}^{r-1}H_{\bar{\partial}}^{p-i,q-i}(Z\cap U)\rightarrow H_{\bar{\partial}}^{p,q}(\widetilde{U}),
\end{displaymath}
where $H^{*,*}_{\bar{\partial}}(Z\cap U)=0$ for $Z\cap U=\emptyset$.
Denote by $\mathcal{P}(U)$ the statement that $\Phi^{p,q}_U$ are isomorphisms for all $p$, $q$. The theorem is equivalent to say that $\mathcal{P}(X)$ holds.
We need to check that $\mathcal{P}$ satisfies the three conditions in Lemma \ref{glued}. Obviously, $\mathcal{P}$ satisfies the disjoint condition.

For open sets $V\subseteq U$, denote by $\rho^U_V:\mathcal{F}^{p,q}(U)\rightarrow \mathcal{F}^{p,q}(V)$ the restriction of the sheaf $\mathcal{F}^{p,q}$ and by $j^U_V:\mathcal{D}^{\prime p,q}(\widetilde{U})\rightarrow \mathcal{D}^{\prime p,q}(\widetilde{V})$ the  restriction of currents. By Lemma \ref{com},  $j^U_V\circ\varphi_U=\varphi_V\circ\rho^U_V$. Given $p$, for any open subsets $U$, $V$ in $X$, there is a commutative diagram of  complexes
\begin{displaymath}
\xymatrix{
0\ar[r]&\mathcal{F}^{p,\bullet}(U\cup V)\ar[d]^{\varphi^{p,\bullet}_{U\cup V}}\quad \ar[r]^{(\rho_U^{U\cup V},\rho_V^{U\cup V})\quad\quad} & \quad\mathcal{F}^{p,\bullet}(U)\oplus \mathcal{F}^{p,\bullet}(V)\ar[d]^{(\varphi^{p,\bullet}_{U},\varphi^{p,\bullet}_{V})}\quad\ar[r]^{\quad\quad\rho_{U\cap V}^U-\rho_{U\cap V}^V}& \quad\mathcal{F}^{p,\bullet}(U\cap V) \ar[d]^{\varphi^{p,\bullet}_{U\cap V}}\ar[r]& 0\\
0\ar[r]&\mathcal{D}^{\prime p,\bullet}(\widetilde{U}\cup \widetilde{V})     \quad  \ar[r]^{(j_U^{U\cup V},j_V^{U\cup V})\quad\quad}& \quad\mathcal{D}^{\prime p,\bullet}(\widetilde{U})\oplus \mathcal{D}^{\prime p,\bullet}(\widetilde{V})    \quad\ar[r]^{\quad\quad j_{U\cap V}^U-j_{U\cap V}^ V} & \quad\mathcal{D}^{\prime p,\bullet}(\widetilde{U}\cap \widetilde{V})     \ar[r]& 0 .}
\end{displaymath}
The two rows are both exact sequences of complexes. For convenience, denote
\begin{displaymath}
L^{p,q}(U)=H_{\bar{\partial}}^{p,q}(U)\oplus \bigoplus_{i=1}^{r-1}H_{\bar{\partial}}^{p-i,q-i}(Z\cap U).
\end{displaymath}
Therefore, we have a commutative diagram
\begin{displaymath}
\tiny{\xymatrix{
    \cdots\ar[r]&L^{p,q-1}(U\cap V)\ar[d]^{\Phi^{p,q-1}_{U\cap V}}\ar[r]&L^{p,q}(U\cup V) \ar[d]^{\Phi^{p,q}_{U\cup V}} \ar[r]& L^{p,q}(U)\oplus L^{p,q}(V) \ar[d]^{(\Phi^{p,q}_U,\Phi^{p,q}_V)}\ar[r]&  L^{p,q}(U\cap V)\ar[d]^{\Phi^{p,q}_{U\cap V}}\ar[r]&L^{p,q+1}(U\cup V)\ar[d]^{\Phi^{p,q+1}_{U\cup V}}\ar[r]&\cdots\\
 \cdots\ar[r]&H_{\bar{\partial}}^{p,q-1}(\widetilde{U}\cap \widetilde{V})     \ar[r] & H_{\bar{\partial}}^{p,q}(\widetilde{U}\cup \widetilde{V})\ar[r]&H_{\bar{\partial}}^{p,q}(\widetilde{U})\oplus H_{\bar{\partial}}^{p,q}(\widetilde{V})       \ar[r]& H_{\bar{\partial}}^{p,q}(\widetilde{U}\cap \widetilde{V})     \ar[r] & H_{\bar{\partial}}^{p,q+1}(\widetilde{U}\cup \widetilde{V})\ar[r]&\cdots}}
\end{displaymath}
 of long exact sequences. If $\Phi^{p,q}_U$, $\Phi^{p,q}_V$ and $\Phi^{p,q}_{U\cap V}$ are isomorphisms for all $p$, $q$, then so are $\Phi^{p,q}_{U\cup V}$ for all $p$, $q$ by the five-lemma.   Thus $\mathcal{P}$  satisfies the Mayer-Vietoris condition.

Let $\mathcal{U}$ be a basis of the topology of $X$ such that every $U\in \mathcal{U}$ is Stein. For $U_1$, $\ldots$, $U_l\in\mathcal{U}$,  $\bigcap\limits_{i=1}^lU_i\cap Z$ is Stein. By Proposition \ref{special}, $\Phi_{U_1\cap\ldots\cap U_l}$ is an isomorphism, so $\mathcal{P}$ satisfies the local condition.

We complete the proof.
\end{proof}
\begin{rem}
Grivaux, J. \cite[Proposition 1]{Gri} proved the self-intersection formula for Deligne cohomology on compact complex manifolds by the deformation to the normal cone.
His proof is valid for Dolbeault cohomology and hence Theorem \ref{main} can be directly proved as Proposition \ref{special} for compact cases.
Notice that the compactness is necessary there.
\end{rem}

\begin{rem}\label{inverse}
After the early version \cite{M2} of this paper was posed on arXiv, %Stelzig, J. \cite[]{} implicitly and
Rao, S., Yang, S. and Yang, X.-D. \cite[Theorem 1.2]{RYY2}
%explicitly
gave another expression of the blow-up formula on compact complex manifolds.
If the self-intersection formula holds, their isomorphism is inverse to (\ref{b-u-m}), %the one in Proposition \ref{main},
which can be seen from the proof of Theorem \ref{special}.
For more details, refer to \cite[Propositions 6.3, 6.4]{M3}.
\end{rem}

\section{Applications}
\subsection{Strongly $q$-complete manifolds}
Recall that, a complex manifold $X$ is called a \emph{strongly $q$-complete manifold}, if there exists a smooth strongly $q$-convex exhaustion function (cf. \cite[IX, Definition (2.5)]{D}) on it. The following theorem is due to Andreotti, A. and Grauert, H. (cf. \cite[IX, Corollary (4.11)]{D}).
\begin{thm}\label{A-G}
Let $X$ be a strongly $q$-complete manifold. Then $H^p(X,\mathcal{F})=0$ for  any analytic coherent sheaf $\mathcal{F}$ on $X$ and all $p\geq q$.
\end{thm}

%\begin{prop}
%Let $E$ be a holomorphic fiber bundle over a connected complex manifold $X$. Assume that there exist classes $e_1$, $\dots$, $e_r$ of pure degrees in $ H_{\bar{\partial}}^{*,*}(E)$, such that,  their restrictions $e_1|_{E_x}$, $\dots$, $e_r|_{E_x}$ to $E_x$ freely linearly generate $H_{\bar{\partial}}^{*,*}(E_x)$ for every $x\in X$. Denote the degree of $e_i$ by $(u_i,v_i)$ for $1\leq i\leq r$ and set $v=max\{v_1,\ldots,v_r\}$. Then $E$ is not strongly $q$-complete for $1\leq q\leq v-1$.
%\end{prop}
%\begin{proof}
%Suppose that $v=v_{i_0}$ for some $i_0$.
%By Theorem \ref{L-H}, $H^{v_{i_0}}(E,\Omega_E^{u_{i_0}})\supseteq H^{0}(X,\mathcal{O}_X)\neq 0$.
%By Theorem \ref{A-G}, we get the corollary.
%\end{proof}

\begin{prop}
A projective bundle associated to a holomorphic vector bundle of rank $r\geq 2$ is not strongly $q$-complete for $1\leq q\leq r-1$.
\end{prop}
\begin{proof}
Let $\mathbb{P}(E)$ be the projective bundle associated to a holomorphic vector bundle $E$ of rank $r$.
By Corollary \ref{poj-D}, $H^{r-1}(\mathbb{P}(E),\Omega^{r-1}_{\mathbb{P}(E)})\supseteq H^{0}(X,\mathcal{O}_X)=\mathcal{O}(X)\neq 0$. So $\mathbb{P}(E)$ is not strongly $q$-complete for $1\leq q\leq r-1$ by Theorem \ref{A-G}.
\end{proof}

\begin{prop}
Let $X$ be a connected complex manifold  and $Z$ a connected complex submanifold of $X$ with codimension $r\geq2$. Then the blow-up $\widetilde{X}_Z$ of $X$ along $Z$ is not strongly $q$-complete for $1\leq q\leq r-1$.
\end{prop}
\begin{proof}
By Theorem \ref{main}, $H^{r-1}(\widetilde{X}_Z,\Omega^{r-1}_{\widetilde{X}_Z})\supseteq H^{0}(Z,\mathcal{O}_Z)=\mathcal{O}(Z)\neq0$. So $\widetilde{X}_Z$ is not strongly $q$-complete for $1\leq q\leq r-1$ by Theorem \ref{A-G}.
\end{proof}

\subsection{$\partial\bar{\partial}$-lemma}
The \emph{$\partial\bar{\partial}$-lemma} on a compact complex manifold $X$ refers
to that for every pure-type $d$-closed form on $X$, the properties of
$d$-exactness, $\partial$-exactness, $\bar{\partial}$-exactness and
$\partial\bar{\partial}$-exactness are equivalent. The well-known fact is that compact K\"{a}hler manifolds satisfy the $\partial\bar{\partial}$-lemma. The complex manifolds in the Fujiki class  $\mathscr{C}$ satisfy the $\partial\bar{\partial}$-lemma (\cite{DGMS,Va}) and are not K\"{a}hlerian in general.
There also exist many examples which are not in $\mathscr{C}$, see \cite{AK,AK2,ASTT,K}.
As we know (\cite[(5.16)]{DGMS}), $X$ satisfies the $\partial\bar{\partial}$-lemma, if and only if,  the natural map $H_{BC}^{*,*}(X)\rightarrow H_{\bar{\partial}}^{*,*}(X)$ induced by the identity is an isomorphism.

We recall some notations of the double complex and its cohomologies, see \cite[Section 1]{St2} or  \cite{AK, St1, St3}. \emph{A bounded double complex over $\mathbb{C}$ with real structure} refers to the quadruple $(K^{*,*},\partial_1, \partial_2,\sigma)$, where

$(1)$ $K^{*,*}$ is a bigraded $\mathbb{C}$-vector space for all  $(p,q)\in\mathbb{Z}$, which satisfies that $K^{p,q}=0$ except for finitely many $(p,q)\in\mathbb{Z}$,

$(2)$ $\partial_1$ and $\partial_2:K^{*,*}\rightarrow K^{*,*}$ are $\mathbb{C}$-linear maps of bidegree $(1,0)$ and $(0,1)$ such that $\partial_i^2=0$ for $i=1,2$ and $\partial_1\circ\partial_2+\partial_2\circ\partial_1=0$,

$(3)$ $\sigma:K^{*,*}\rightarrow K^{*,*}$ is a  conjugation-antilinear involution, which satisfies that $\sigma (K^{p,q})=K^{q,p}$ and $\sigma\partial_1\sigma=\partial_2$.\\
A \emph{morphism} of bounded double complexes over $\mathbb{C}$ with real structures means a $\mathbb{C}$-linear map of underlying $\mathbb{C}$-vector spaces compatible  with the bigrading, differentials and real structure.

Let $\partial^{p,q}_1:K^{p,q}\rightarrow K^{p+1,q}$ and $\partial^{p,q}_2:K^{p,q}\rightarrow K^{p,q+1}$ be the restrictions of $\partial_1$ and $\partial_2$ respectively. Define the  \emph{Dolbeault cohomology} as $H^{p,q}_{\partial_2}(K)=H^{q}(K^{p,\bullet},\partial_2)$ and the \emph{Bott-Chern cohomology} as $H^{p,q}_{BC}(K)=\frac{\textrm{ker}\partial^{p,q}_1\cap\textrm{ker}\partial^{p,q}_2}{\textrm{im}\partial^{p-1,q}_1\circ\partial^{p-1,q-1}_2}$. A morphism of bounded double complexes over $\mathbb{C}$ with real structures is called an \emph{$E_1$-isomorphism}, if it induces an isomorphism on the Dolbeault cohomology.
%The filtration $F^pK^\bullet=\bigoplus\limits_{\substack{r+s=\bullet\\r\geq p}}K^{r,s}$ for all $p$ on the complex $K^\bullet$ associated to $(K^{\bullet,\bullet},\partial_1, \partial_2,\sigma)$
%gives rise to a spectral sequence $(E_r^{p,q}, F^p H^k(K))$. We  call it the \emph{Fr\"{o}licher spectral sequence} and  $F^\bullet H^k(K)$ the \emph{Hodge filtration} of the de Rham cohomology $H^k(K)$ of $(K^{\bullet,\bullet},\partial_1, \partial_2,\sigma)$ . Evidently, $E_1^{p,q}=H_{\partial_2}^{p,q}(K)$.

\begin{thm}[{\cite[Lemma 4]{St2}\cite[Lemma 1.10]{St1}\cite{AK}}]\label{Stelzig}
Any $E_1$-isomorphism induces  an isomorphism on the Bott-Chern cohomology.
\end{thm}
Now we consider the $\partial\bar{\partial}$-lemma on fiber bundles and blow-ups.
\subsubsection{Fiber bundles}
\begin{thm}\label{L-H1}
Let $\pi:E\rightarrow X$ be a holomorphic fiber bundle  over a connected compact complex manifold $X$. Assume that there exist  $\bar{\partial}$-closed  forms $t_1$, $\ldots$, $t_r$ of pure degrees satisfying the following conditions:

$(1)$ $t_{i+s}=\bar{t}_i$ for $1\leq i\leq s$ and  $t_i=\bar{t}_{i}$ for $2s+1\leq i\leq r$,

$(2)$ the restrictions  of Dolbeault  classes $[t_1]_{\bar{\partial}}$, $\ldots$, $[t_r]_{\bar{\partial}}$ to $E_x$  is a basis of $H^{*,*}(E_x)$ for every $x\in X$. \\
Then $E$ satisfies the $\partial\bar{\partial}$-lemma if and only if $X$ does.
\end{thm}
\begin{proof}
Set $t_i\in A^{u_i,v_i}(E)$ for $1\leq i\leq r$. Clearly, $(u_{s+i},v_{s+i})=(v_i,u_i)$ for $1\leq i\leq s$ and $u_i=v_i$ for $2s+1\leq i\leq r$. Define a bigraded $\mathbb{C}$-vector space
\begin{displaymath}
K^{*,*}=\bigoplus_{i=1}^rA^{*,*}(X)[-u_i,-v_i],
\end{displaymath}
where $[-u_i,-v_i]$ denotes the degree shift for $1\leq i\leq r$. Define two $\mathbb{C}$-linear maps
\begin{displaymath}
\partial_1=\bigoplus_{i=1}^r\partial \mbox{ , } \partial_2=\bigoplus_{i=1}^r\bar{\partial}:K^{*,*}\rightarrow K^{*,*},
\end{displaymath}
where $\partial$ and $\bar{\partial}$ denote the partial differential operators  on $A^{*,*}(X)$. Define $\sigma:K^{*,*}\rightarrow K^{*,*}$ as
%\begin{displaymath}
%\begin{aligned}
%&\sigma(\sum\limits_{p,q\geq 0}\alpha^{p-u_1,q-v_1},\sum\limits_{p,q\geq 0}\alpha^{p-u_2,q-v_2},\ldots,\sum\limits_{p,q\geq 0}\alpha^{p-u_r,q-v_r})\\
%=&(\sum\limits_{p,q\geq 0}\overline{\alpha^{q-u_{s+1},p-v_{s+1}}},\sum\limits_{p,q\geq 0}\alpha^{p-u_2,q-v_2},\ldots,\sum\limits_{p,q\geq 0}\alpha^{p-u_r,q-v_r}).
%\end{aligned}
%\end{displaymath}
\begin{displaymath}
\sigma(\alpha_1,\alpha_2,\ldots,\alpha_r)
=(\bar{\alpha}_{s+1},\bar{\alpha}_{s+2},\ldots,\bar{\alpha}_{2s},\bar{\alpha}_1,\bar{\alpha}_2,\ldots,\bar{\alpha}_s,\bar{\alpha}_{2s+1},\ldots,\bar{\alpha}_r),
\end{displaymath}
where $\alpha_i\in A^{*,*}(X)[-u_i,-v_i]$ for $1\leq i\leq r$. The quadruple $(K^{*,*},\partial_1, \partial_2,\sigma)$ is a bounded double complex over $\mathbb{C}$ with real structure. Define $f:K^{*,*}\rightarrow A^{*,*}(E)$ as
\begin{displaymath}
(\alpha_1,\ldots,\alpha_r)\mapsto\sum_{i=1}^r\pi^*\alpha_i\wedge t_i,
\end{displaymath}
which is a morphism of bounded double complexes over $\mathbb{C}$ with real structures.

% Since  $[t_1]_{\overline{\partial}}|_{E_x},\dots,[t_r]_{\overline{\partial}}|_{E_x}$ freely linearly generate $H_{\overline{\partial}}^{**}(E_x)=H_{\textrm{BC}}^{**}(E_x)$ for every $x\in X$,  By \cite[Theorem 1.2]{M2},
%\begin{displaymath}
%\sum\limits_{i=1}^r\pi^*(\bullet)\cup [t_i]_{\overline{\partial}}:\bigoplus\limits_{i=1}^rH^{p-u_i,q-u_i}(X) \rightarrow H^{p,q}(E)
%\end{displaymath}
%is an isomorphism, i.e., $f$ is an $E_1$-isomorphism. By Theorem \ref{Stelzig},
%\begin{displaymath}
%\sum\limits_{i=1}^r\pi^*(\bullet)\cup [t_i]_{BC}:\bigoplus\limits_{i=1}^rH_{BC}^{p-u_i,q-v_i}(X) \tilde{\rightarrow} H_{BC}^{p,q}(E)
%\end{displaymath}
%is an isomorphism for any $p$, $q$.
Consider the commutative diagram
\begin{displaymath}
\xymatrix{
  \bigoplus\limits_{i=1}^{r-1}H^{*,*}_{BC}(X)[-u_i,-v_i]\ar[d]_{} \qquad \ar[r]^{\qquad\sum\limits_{i=1}^r\pi^*(\bullet)\cup [t_i]_{BC}}  &\quad H^{*,*}_{BC}(E) \ar[d]^{} \\
  \bigoplus\limits_{i=1}^{r-1}H^{*,*}_{\bar{\partial}}(X)[-u_i,-v_i]\qquad \ar[r]^{\qquad\sum\limits_{i=1}^r\pi^*(\bullet)\cup [t_i]_{\bar{\partial}}}  & \quad H^{*,*}_{\bar{\partial}}(E),}
\end{displaymath}
where the vertical maps are induced by identities. By Theorem \ref{L-H}, the bottom map is an isomorphism, i.e., $f$ is an $E_1$-isomorphism. By Theorem \ref{Stelzig}, the top map is isomorphic. We easily get the theorem by the commutative diagram.
\end{proof}

\begin{rem}
In the above proof, we get the Leray-Hirsch theorem of Bott-Chern cohomology, whose conditions are stronger than those of Dolbeault cohomology.
For instance, let $X$ be any complex manifold and $H_\alpha$ the Hopf surface (see Remark \ref{Hopf}).
View $X\times H_\alpha$ as a holomorphic fiber bundle over  $X$ with fiber $H_\alpha$.
On the one hand, the conditions in  Theorem \ref{L-H1} imply  symmetries of the Hodge numbers of fibers, which are not satisfied by $H_\alpha$ (\cite[V, Proposition (18.1)]{BHPV}).
On the other hand, suppose that $h_1$,  $h_2$, $h_3$,  $h_4$ is a basis of $H_{\bar{\partial}}^{*,*}(H_\alpha)$, then $pr_2^*h_1$, $pr_2^*h_2$, $pr_2^*h_3$, $pr_2^*h_4$ satisfy the condition in  Theorem \ref{L-H}, where  $pr_2$ is the projection from $X\times H_\alpha$ onto  $H_\alpha$.
\end{rem}

The $\partial\bar{\partial}$-lemma on projective bundles was studied  in \cite{ASTT,M5,M6}. We generalize it as follows.

\begin{cor}\label{flagbun}
Let $F$ be a flag bundle associated to a holomorphic vector bundle on  a connected compact  complex manifold $X$. Then $F$ satisfies the $\partial\bar{\partial}$-lemma, if and only if, $X$ does.
\end{cor}
\begin{proof}
Suppose that $E$ is a holomorphic vector bundle with rank $n$ over $X$ and $n_1$, $\ldots$, $n_r$ is a sequence of positive integers with $\sum\limits_{i=1}^rn_i=n$, such that the fiber $F_x$ of $F$ over $x\in X$ is the flag variety $Fl(n_1,\ldots,n_r)(E_x)=$  %$F(n_1,\ldots,n_r)(E_x)=\{(F_1,\ldots,F_{r-1})|0=F_0\subsetneq F_1\subsetneq F_2\subsetneq \ldots \subsetneq F_{r-1}\subsetneq F_r=E_x, \mbox{ where } F_i \mbox{ is a complex vector space with }\textrm{dim}_{\mathbb{C}}F_i=\sum\limits_{j=1}^{i}n_j \mbox{  for every } i\}$.
\begin{displaymath}
\begin{aligned}
\{(V_0,V_1,\ldots,V_{r-1},V_r)|&0=V_0\subsetneq V_1\subsetneq V_2\subsetneq \ldots \subsetneq V_{r-1}\subsetneq V_r=E_x, \mbox{ where } V_i \mbox{ is a } \\
& \mbox{ complex vector space with dimension } \sum\limits_{j=1}^{i}n_j \mbox{  for } 1\leq i\leq r\}.
\end{aligned}
\end{displaymath}
For $0\leq i\leq r$, let $E_i$ be the  universal subbundle over $F$ whose fiber over the point $(V_0,V_1,\ldots,V_{r-1},V_r)$ is $V_i$.
Notice that $E_0=0$ and $E_r=\pi^*E$, where $\pi$ is the projection from $F$ onto $X$.
Denote by $E^{(i)}=E_i/E_{i-1}$ the successive  universal quotient bundles  and by $t_j^{(i)}\in A^{j,j}(F)$ a $j$-th Chern form of $E^{(i)}$ for $1\leq i\leq r$.
For any $x\in X$, the restrictions $t_{j_i}^{(i)}|_{F_x}$ ($1\leq i\leq r$, $1\leq j_i\leq n_i$) to $F_x$ are  Chern forms of successive  universal quotient bundles of the flag manifold $F_x$.
As we know, there exists the monomials $Q_i(T_1^{(1)}, \ldots,  T_{n_1}^{(1)}, \ldots, T_1^{(r)}, \ldots, T_{n_r}^{(r)})$ for $1\leq i\leq k$ such that
\begin{displaymath}
Q_i([t_{1}^{(1)}|_{F_x}]_{\bar{\partial}}, \ldots,  [t_{n_1}^{(1)}|_{F_x}]_{\bar{\partial}}, \ldots, [t_1^{(r)}|_{F_x}]_{\bar{\partial}}, \ldots, [t_{n_r}^{(r)}|_{F_x}]_{\bar{\partial}}), \mbox{  } 1\leq i\leq k,
\end{displaymath}
is a basis of the Dolbeault cohomology $H_{\bar{\partial}}^{*,*}(F_x)$ of the flag manifold $F_x$.
Denote $t_i=Q_i(t_1^{(1)}, \ldots,  t_{n_1}^{(1)}, \ldots, t_1^{(r)}, \ldots, t_{n_r}^{(r)})$ for $1\leq i\leq k$.
Clearly,  $t_i$ are all real and $\bar{\partial}$-closed on $F$. Then $t_1$, $\ldots$, $t_k$ satisfy the conditions in Theorem \ref{L-H1}.
Thus we get the corollary.
\end{proof}

We proved the following corollary in \cite{M6} by  the Deligne-Griffiths-Morgan-Sullivan criterion \cite[(5.21)]{DGMS}   and  give a new proof with a  viewpoint of fiber bundles now.
\begin{cor}[{\cite[Theorem 1.5]{M6}}]\label{product}
Let $X$ and $Y$ be connected compact complex manifolds. Then  $X\times Y$ satisfies the $\partial\bar{\partial}$-lemma if and only if $X$ and $Y$ both do.
\end{cor}
\begin{proof}
Denote by $pr_1$ and $pr_2$ the projections from $X\times Y$ onto $X$ and $Y$ respectively.

Let $X$ and $Y$ both satisfy the $\partial\bar{\partial}$-lemma.  The identity induces an isomorpism $H^{*,*}_{BC}(Y)\rightarrow H^{*,*}_{\bar{\partial}}(Y)$, hence we can choose $\bar{\partial}$-closed forms $h_1$, $\ldots$, $h_r\in A^{*,*}(Y)$ of pure degrees such that $h_{i+s}=\bar{h}_i$ for $1\leq i\leq s$,  $h_i=\bar{h}_{i}$ for $2s+1\leq i\leq r$ and the Dolbeault classes $[h_1]_{\bar{\partial}}$, $\ldots$, $[h_r]_{\bar{\partial}}$ is a basis of $H^{*,*}_{\bar{\partial}}(Y)$. Then $t_i=pr_2^*h_i$ for $1\leq i\leq r$ satisfy the conditions in Theorem \ref{L-H1}, where we view $pr_1:X\times Y\rightarrow X$ as a holomorphic fiber bundle over $X$.
Thus $X\times Y$ satisfies the $\partial\bar{\partial}$-lemma.

Let $X\times Y$ satisfy the $\partial\bar{\partial}$-lemma. We only prove that $X$ satisfies the $\partial\bar{\partial}$-lemma. Since the identity induced an isomorphism $H^{*,*}_{BC}(X\times Y)\tilde{\rightarrow} H^{*,*}_{\bar{\partial}}(X\times Y)$, there exists  $\bar{\partial}$-closed forms  $t_1$, $\ldots$,  $t_r$ of pure degrees on $X\times Y$ satisfying that

$(1)$ $t_i\in\bigoplus\limits_{ p>q}A^{p,q}(X\times Y)$ and $t_{i+s}=\bar{t}_i$ for $1\leq i\leq s$,

$(2)$ $t_{i}\in\bigoplus\limits_{p=q}A^{p,q}(X\times Y)$ and $t_{i}=\bar{t}_i$ for $2s+1\leq i\leq r$,

$(3)$ $[t_1]_{\bar{\partial}}$, $\ldots$,  $[t_r]_{\bar{\partial}}$ is a basis of $H^{*,*}_{\bar{\partial}}(X\times Y)$.\\
For $x\in X$, denote by $i_x:\{x\}\times Y\rightarrow  X\times Y$ the inclusion. Then $pr_1\circ i_x$ is constant and $pr_2\circ i_x:\{x\}\times Y\rightarrow Y$ is the identity.
By the K\"{u}nneth formula (\ref{Kun}), $H^{*,*}_{\bar{\partial}}(\{x\}\times Y)=i_x^*H^{*,*}_{\bar{\partial}}(X\times Y)$, so $[i^*_{x}t_1]_{\bar{\partial}}$, $\ldots$,  $[i^*_{x}t_r]_{\bar{\partial}}$ linearly generate $H^{*,*}_{\bar{\partial}}(\{x\}\times Y)$.
Fix a point $x_0\in X$, we may assume that $[i^*_{x_0}t_1]_{\bar{\partial}}$, $\ldots$,  $[i^*_{x_0}t_{s_0}]_{\bar{\partial}}$ and $[i^*_{x_0}t_{2s+1}]_{\bar{\partial}}$, $\ldots$, $[i^*_{x_0}t_{r_0}]_{\bar{\partial}}$ are maximally linearly independent in $[i^*_{x_0}t_1]_{\bar{\partial}}$, $\ldots$,  $[i^*_{x_0}t_{s}]_{\bar{\partial}}$ and $[i^*_{x_0}t_{2s+1}]_{\bar{\partial}}$, $\ldots$, $[i^*_{x_0}t_{r}]_{\bar{\partial}}$ respectively. Then $[i^*_{x_0}t_1]_{\bar{\partial}}$, $\ldots$, $[i^*_{x_0}t_{s_0}]_{\bar{\partial}}$, $[i^*_{x_0}t_{s+1}]_{\bar{\partial}}$, $\ldots$, $[i^*_{x_0}t_{s+s_0}]_{\bar{\partial}}$, $[i^*_{x_0}t_{2s+1}]_{\bar{\partial}}$, $\ldots$, $[i^*_{x_0}t_{r_0}]_{\bar{\partial}}$ is a basis of $H^{*,*}_{\bar{\partial}}(\{x_0\}\times Y)$. For any $x\in X$, $[i^*_{x}t_i]_{\bar{\partial}}=id_x^*[i^*_{x_0}t_i]_{\bar{\partial}}$, where $id_x:\{x\}\times Y\rightarrow \{x_0\}\times Y$ induced by the identity. So $t_1$, $\ldots$, $t_{s_0}$, $t_{s+1}$, $\ldots$, $t_{s+s_0}$, $t_{2s+1}$, $\ldots$, $t_{r_0}$ satisfy the conditions in Theorem \ref{L-H1}, where $pr_1:X\times Y\rightarrow X$ is viewed as a holomorphic fiber bundle over $X$.
Thus $X$ satisfies the $\partial\bar{\partial}$-lemma.
\end{proof}

\subsubsection{Blow-ups}
Let $\pi:\widetilde{X}_Z\rightarrow X$ be the blow-up of a connected complex manifold $X$ along a connected complex submanifold $Z$ of codimension $\geq 2$. Combining Theorems  \ref{main} and \ref{Stelzig}, we can get a similar blow-up formula for Bott-Chern cohomology, which can be used to  prove that the $\partial\bar{\partial}$-lemma holds on $\widetilde{X}_Z$ if and only if it holds on $X$ and $Z$, refer to \cite[Section 4]{M6}.  See also \cite{YY,ASTT,St1,M5} and \cite[Section 3]{M6} for other proofs.
% \cite[Theorem 1.3]{YY}\cite[Theorem 5]{ASTT}\cite[Proposition 2.3]{M5}\cite[Section 3]{M6}.

%==============================================================

\end{document}